\documentclass[a4paper,12pt]{article}

\usepackage{amsmath}
\usepackage{amsthm}

\newtheorem{proposition}{Proposition}[section]

\usepackage[vvarbb]{newtx}
\usepackage{bm}
\usepackage{cite}
\usepackage{mathtools}
\mathtoolsset{showonlyrefs=true}
\usepackage{algorithm}
\usepackage{algpseudocode}
\algrenewcommand\textproc{\texttt}
\usepackage[truedimen,a4paper,margin=20mm]{geometry}

\usepackage{hyperref}
\hypersetup{%
    pdftitle={Computing the matrix exponential with the double exponential formula},
    pdfauthor={F. Tatsuoka, T. Sogabe, T. Kemmochi, S.-L. Zhang},
    pdfkeywords={matrix function, matrix exponential, numerical quadrature, double exponential formula},
    colorlinks, citecolor=blue
}
\newcommand{\bbC}{\mathbb{C}}

\newcommand{\bbZ}{\mathbb{Z}}
\newcommand{\bmb}{\bm{b}}
\newcommand{\bmu}{\bm{u}}
\newcommand{\bmx}{\bm{x}}
\newcommand{\calI}{\mathcal{I}}
\newcommand{\calN}{\mathcal{N}}
\newcommand{\calW}{\mathcal{W}}
\newcommand{\calO}{\mathcal{O}}
\newcommand{\rme}{\mathrm{e}}
\newcommand{\rmi}{\mathrm{i}}
\newcommand{\sfH}{\mathsf{H}}

\newcommand{\lmdr}{\lambda_{\mathrm{right}}}
\newcommand{\sgm}{\sigma}
\newcommand{\eps}{\epsilon}
\newcommand{\tldA}{\tilde{A}}
\newcommand{\tldX}{\tilde{X}}
\newcommand{\hmin}{h_{\mathrm{min}}}
\newcommand{\dd}{\,\mathrm{d}}
\newcommand{\norm}[1]{\left\|{#1}\right\|}

\renewcommand{\Re}{\mathrm{Re}}
\renewcommand{\Im}{\mathrm{Im}}

\title{Computing the matrix exponential with the double exponential formula}
\author{%
  Fuminori Tatsuoka\thanks{Department of Applied Physics, Graduate School of Engineering, Nagoya University, Furo-cho, Chikusa-ku, Nagoya 464-8603, Japan (\texttt{\{f-tatsuoka,sogabe,kemmochi,zhang\}@na.nuap.nagoya-u.ac.jp}).}
  \and Tomohiro Sogabe\footnotemark[1]
  \and Tomoya Kemmochi\footnotemark[1]
  \and Shao-Liang Zhang\footnotemark[1]
}

\begin{document}

\maketitle
\begin{abstract}
  This paper considers the computation of the matrix exponential $\rme^A$ with numerical quadrature.
  Although several quadrature-based algorithms have been proposed, they focus on (near) Hermitian matrices.
  In order to deal with non-Hermitian matrices, we use another integral representation including an oscillatory term and consider applying the double exponential (DE) formula specialized to Fourier integrals.
  The DE formula transforms the given integral into another integral whose interval is infinite, and therefore it is necessary to truncate the infinite interval.
  In this paper, to utilize the DE formula, we analyze the truncation error and propose two algorithms.
  The first one approximates $\rme^A$ with the fixed mesh size which is a parameter in the DE formula affecting the accuracy.
  Second one computes $\rme^A$ based on the first one with automatic selection of the mesh size depending on the given error tolerance.
\end{abstract}

\section{Introduction}
The exponential of a matrix $A\in\bbC^{n\times n}$ is defined as
\begin{align}
  \rme^A := I + A + \frac{1}{2!}A^2 + \frac{1}{3!}A^3 + \cdots,
\end{align}
and it arises in several situations in scientific computing.
One of the applications is the exponential integrator, a class of numerical solvers for stiff ordinary differential equations, see, e.g., \cite{hochbruck2010exponential}.
In recent years, $\rme^A$ also arises in the analysis of directed graphs \cite{delacruzcabrera2019analysis}.
Thus, several classes of computational methods of $\rme^A$ and $\rme^A\bmb ~ (\bmb\in\bbC^n)$ have been proposed.
For example,
methods based on Pad\'{e} approximation of $\rme^z$ at $z=0$ \cite{al-mohy2011computing,fasi2019arbitrary,al-mohy2009new},
methods based on projections onto Krylov-like subspaces \cite{saad1992analysis,gockler2014uniform},
methods based on the best approximation of $\rme^z$ on $(-\infty,0]$ \cite{schmelzer2007evaluating},
and methods based on numerical quadrature \cite{weideman2007parabolic,dingfelder2015improved}.
In general, quadrature-based methods have two advantages:
these algorithms can compute $\rme^A\bmb$ without computing $\rme^A$ itself when $A$ is large and either sparse or structured, and it is possible to easily make these algorithms parallel in the sense that the integrand can be computed independently on each abscissa.
See \cite[Sect. 18]{trefethen2014exponentially} for more details.

These quadrature-based algorithms\cite{weideman2007parabolic,dingfelder2015improved} focus on (nearly) Hermitian matrices.
In detail, they are proposed for the efficient computation of Bromwich integrals whose integrand has singularities on or near the negative real axis.
For the computation of $\rme^A$, (all) the eigenvalues of $A$ are the singularities of the integrand.
Hence, when $A$ has eigenvalues with large imaginary parts, the singularities are also located far from the negative real axis, and these algorithms would be inaccurate or may not converge.

The motivation of this study is to construct quadrature-based algorithms that can be used for non-Hermitian matrices.
Here, we consider another integral representation
\begin{align}\label{eq:ir_rme}
  \rme^A = \frac{2}{\pi} \int_0^\infty x \sin(x) (x^2I + A^2)^{-1} \dd{x}.
\end{align}
The integral representation \eqref{eq:ir_rme} holds when all the eigenvalues of $A$ lie in the open left half plane $\{z\in\bbC\colon \Re(z) < 0\}$, but this condition can be assumed without loss of generalities because $\rme^{A + sI} = \rme^s \rme^A\ (s\in\bbC)$.

Since the integrand in \eqref{eq:ir_rme} includes an oscillatory term $\sin(x)$ and the interval is the half infinite, it is difficult to compute \eqref{eq:ir_rme} by using typical quadrature formula such as Gaussian quadrature.
To compute \eqref{eq:ir_rme}, we consider applying the double exponential (DE) formula specialized to Fourier integrals \cite{ooura1999robust}.

In this paper, we propose algorithms using the DE formula.
The DE formula transforms a given integral into another integral suited for the trapezoidal rule.
Because the transformed integral interval is infinite, we need to truncate the infinite sum of the discretized integral appropriately into a finite one.
Thus, we propose a truncation method of the infinite sum based on error analysis.
In addition, we show an automatic quadrature algorithm that selects the mesh size of the trapezoidal rule depending on the given tolerance.

The organization of this paper is as follows:
The DE formula used in this paper is introduced in Section \ref{sect:introduction_de}.
In Section \ref{sect:algorithm}, we analyze the truncation error and propose algorithms.
Numerical results are presented in Section \ref{sect:experiment}, and we conclude this paper in Section \ref{sect:conclusion}.

\section{The DE formula for Fourier integrals} \label{sect:introduction_de}
The DE formula exploits the fact that the trapezoidal rule for the integrals of analytic functions on the real line converges exponentially, see, e.g., \cite{takahasi1974double}.
Several types of change of variables are proposed to deal with different types of integrals.
In \cite{ooura1991double}, integrals form of
\begin{align}\label{eq:int_demo}
  \calI = \int_0^\infty g(x) \sin(x) \dd{x},
\end{align}
are considered, where $g$ be a scalar function decaying polynomially as $x \to \infty$.

The DE formula for \eqref{eq:int_demo} is as follows.
We first select the mesh size $h > 0$ to be used for the trapezoidal rule.
Next, we apply a change of variables $x = x_h(t)$ such that $x_h'(t)$ decays double exponentially as $t \to -\infty$ and $x_h(t) - \pi t/h$ decays double exponentially as $t \to \infty$.
Then, we compute the transformed integral with the trapezoidal rule:
\begin{align}
    \calI = \int_{-\infty}^\infty x_h'(t) g(x_h(t)) \sin(x_h(t)) \dd{t} \approx h \sum_{k=-\infty}^\infty x_h'(kh) g(x_h(kh)) \sin(x_h(kh)).
\end{align}
The summand decays double exponentially as $|k| \to \infty$, and therefore the infinite sum can be approximated by a sum of not too many numbers of the summand.
Examples of such change of variables are
\begin{align}\label{eq:trans_91}
  x_h(t) = \frac{\pi}{h} \frac{t}{1 - \exp(-\alpha \sinh t)}
  \qquad (\alpha = 6)
\end{align}
which is proposed in \cite{ooura1991double},
\begin{align}\label{eq:trans_99}
  &x_h(t) = \frac{\pi}{h} \frac{t}{1 - \exp(-2t - \alpha(1-\rme^{-t}) - \beta(\rme^t-1))}\\
  &\quad \left(\beta = \frac{1}{4}, \quad \alpha = \frac{\beta}{\sqrt{1 + \log(1+\pi/h)/4h}}\right),
\end{align}
which is proposed in \cite{ooura1999robust}.
The implementation notes for \eqref{eq:trans_99} is presented in Appendix \ref{sect:derivative_cv}.

The DE formula can be applied to \eqref{eq:ir_rme} because the integrand is a rational function of $A$ in the sense of \cite[Def.\ 1.2]{higham2008functions}.
See \cite[Sect.\ 1.1]{tatsuoka2021computing} for more details of the discussion.

\section{Computing \texorpdfstring{$\rme^A$}{the matrix exponential} with the DE formula} \label{sect:algorithm}
In this section, we propose algorithms for $\rme^A$ based on the DE formula:
\begin{align}\label{eq:ir_rme_de}
  \rme^A = \int_{-\infty}^\infty F_h(t,A) \dd{t} \approx \sum_{k=l}^r h F_h(kh,A)
\end{align}
where
\begin{align}
  F_h(t,A) := x_h'(t) \sin(x_h(t)) G(x_h(t,A)),
  \qquad G(x,A) := \frac{2}{\pi} x (x^2I + A^2)^{-1}.
\end{align}
We may not write the second argument of $F_h$ and $G$ when it is obvious from the context.

The error of the DE formula can be divided into the discretization error and the truncation error:
\begin{align}
  &\int_{-\infty}^\infty F_h(t,A) \dd{t} - \sum_{k=l}^r h F_h(kh,A)\\
  &\quad = \left(\int_{-\infty}^\infty F_h(t) \dd{t} - \sum_{k=-\infty}^\infty h F_h(kh)\right)
    + \left(\sum_{k=-\infty}^\infty h F_h(kh) - \sum_{k=l}^r h F_h(kh)\right).
\end{align}
In Subsection \ref{sect:truncation_error} we analyze the truncation error for the given interval, and in Subsection \ref{sect:propose_algorithm} we show the two algorithms.

\subsection{Truncation error} \label{sect:truncation_error}
We show an upper bound on the truncation error as follows:
\begin{proposition} \label{thm:trunc_err}
    Suppose that all the eigenvalues of $A\in\bbC^{n\times n}$ lie in the left half plane.
    Let $u(t) = hx_h(t)/\pi - 1$, where $x_h(t)$ is a double exponential transformation for Fourier integrals satisfying $x_h'(t) \le 2\pi/h$.
    For given $l, r \in \bbZ$ with $l < r$ and $x_h(lh) \le 1/\sqrt{2\|A^{-2}\|}$,
    we have
  \begin{align} \label{eq:ub_trunc_err}
      & \norm{\sum_{k=-\infty}^\infty hF_h(kh) - \sum_{k=l}^r hF_h(kh)}\\
      & \quad \le \frac{h}{\pi} \sum_{k=-\infty}^{l-1} x_h'(kh) \\
      & \qquad + 2 \pi\sum_{k=r+1}^\infty ku(kh)\left(\norm{(A+\rmi x_h(kh)I)^{-1}} + \norm{(A-\rmi x_h(kh)I)^{-1}}\right),
  \end{align}
  where $\|\cdot\|$ is any subordinate norm.
\end{proposition}
\begin{proof}
  From the triangle inequality, it holds that
  \begin{align}
    \norm{\sum_{k=-\infty}^\infty hF_h(kh) - \sum_{k=l}^r hF_h(kh)}
    \le \norm{\sum_{k=-\infty}^{l-1} hF_h(kh)} + \norm{\sum_{k=r+1}^\infty hF_h(kh)}.
  \end{align}
  In this proof, we show \eqref{eq:ub_trunc_err} by showing
  \begin{align}
    &\norm{\sum_{k=-\infty}^{l-1} h F_h(kh)} \le \frac{h}{\pi} \sum_{k=-\infty}^{l-1} x_h'(kh), \label{eq:ub_trunc_err_l} \\
    &\norm{\sum_{k=r+1}^\infty h F_h(kh)}\\
    & \quad \le 2\pi \sum_{k=r+1}^\infty ku(kh)\left(\norm{(A+\rmi x_h(kh)I)^{-1}} + \norm{(A-\rmi x_h(kh)I)^{-1}}\right). \label{eq:ub_trunc_err_r}
  \end{align}

  We first show \eqref{eq:ub_trunc_err_l}.
  The assumption $(0 < ) x_h(lh)\le 1/\sqrt{2\|A^{-2}\|}$ gives that for $k \le l$, $\|x_h(kh)^2A^{-2}\|\le 1/2\ (< 1)$.
  Hence, by using the Neumann expansion, we have
  \begin{align}
    &\norm{x_h(kh)^2(x_h(kh)^2I + A^2)^{-1}}
      = \norm{x_h(kh)^2A^{-2} (x_h(kh)^2A^{-2} + I)^{-1}}\\
    & \quad =\norm{\sum_{m=0}^\infty \left[x_h(kh)^2A^{-2}\right]^{m+2}}
    \le \sum_{m=2}^\infty \norm{x_h(kh)^2A^{-2}}^m \le \frac{1}{2}.
  \end{align}
  Therefore, we have
  \begin{align}
    & \norm{\sum_{k=-\infty}^{l-1} h F_h(kh)}\\
    & \quad = \frac{2h}{\pi} \norm{
        \sum_{k=-\infty}^{l-1} \frac{x_h'(kh) \sin(x_h(kh))}{x_h(kh)} x_h(kh)^2 (x_h(kh)^2I+A^2)^{-1}
    }\\
    & \quad \le \frac{h}{\pi} \sum_{k=-\infty}^{l-1} x_h'(kh) \frac{\sin(x_h(kh))}{x_h(kh)}
    \le \frac{h}{\pi} \sum_{k=-\infty}^{l-1} x_h'(kh). \label{eq:tmp001}
  \end{align}

  Next, we show \eqref{eq:ub_trunc_err_r}.
  It is easily seen that
  \begin{align}
    & \norm{\sum_{k=r+1}^\infty h F_h(kh)}\\
    & \quad  \le \frac{h}{\pi} \sum_{k=r+1}^\infty x_h'(kh) \sin(x_h(kh)) \norm{2x_h(kh)(x_h(kh)^2I + A^2)^{-1}}\\
    & \quad = \frac{h}{\pi} \sum_{k=r+1}^\infty x_h'(kh) \sin(x_h(kh)) \left[
        \norm{(A + \rmi x_h(kh)I)^{-1}} + \norm{(A - \rmi x_h(kh)I)^{-1}}
    \right].
  \end{align}
  Here, because $x_h'(kh) \le 2\pi/h$ for $k \ge r$ and
  \begin{align}
    |\sin(x_h(kh))| = |\sin(k\pi(1+u(kh)))| = |\sin(k\pi u(kh))| \le k\pi u(kh),
  \end{align}
  we have \eqref{eq:ub_trunc_err_r}, and \eqref{eq:ub_trunc_err} is proved.
\end{proof}

When the numerical range of $A$ is the subset of the left half plane, the upper bound on the 2-norm of the resolvent $\|(A \pm \rmi x I)^{-1}\|_2$ can be obtained.
Let $\mu$ be the supremum of the real part of the numerical range, i.e., $\mu = \sup_{z \in \calW(A)} \Re(z)$, where $\calW(A) = \{\bmx^\sfH A\bmx \colon \bmx \in\bbC^n,\ \|\bmx\|_2 = 1\}$.
If $\mu < 0$, by using the results in \cite{crouzeix2017numerical}, we have
\begin{align}
  \max_{x \ge 0} \|(A \pm \rmi xI)^{-1}\|_2
  \le \max_{x \ge 0} \left[
  (1+\sqrt{2}) \max_{z\in \calW(A)} \left|\frac{1}{z \pm\rmi x}\right|
  \right] = \frac{1+\sqrt{2}}{|\mu|}.
\end{align}
It may not be easy to compute $\mu$ in a practical situation, and we would use the real part of the right-most eigenvalue as an alternative of $\mu$.

\subsection{Algorithms} \label{sect:propose_algorithm}
In this subsection, we propose two algorithms.
The first one computes $\rme^A$ by using the DE formula with the given mesh size $h$ and the second one computes $\rme^A$ with automatic mesh size selection.

In the first algorithm, we first shift the input matrix $A$ so that all eigenvalues of the shifted matrix lie in the left half plain.
Then, the algorithm selects the truncation point of the discretized integral and calculates the exponential of the shifted matrix.
Finally, we get $\rme^A$ by scaling the exponential of the shifted matrix exponential.
The details of the algorithm are given in Algorithm \ref{alg:1}.

\begin{algorithm}
  \caption{Computing $\rme^A$ with the DE formula with given mesh size}
  \begin{algorithmic}[1]
    \Statex \textbf{Input} $A\in\bbC^{n\times n}$, mesh size $h>0$, tolerance for the truncation error $\eps>0$, shift parameter $\sgm < 0$
    \State Compute the rightmost eigenvalue of $A$, $\lmdr$
    \State $\tldA = A + (\sgm - \lmdr)I$
    \Comment The rightmost eigenvalue of $\tldA$ is $\sgm$.
    \State $l, r = \texttt{GetInterval}(\sigma,\eps,h)$
    \State $\tldX = h\sum_{k=l}^r F_h(kh,\tldA)$
    \Statex \textbf{Output} $X = \rme^{\lmdr - \sgm}\tldX \approx \rme^A$
    \Statex
    \Function{GetInterval}{$\sigma$, $\epsilon, h$}
      \Comment Compute an interval whose truncation error will be smaller than $\eps$
      \State Find the maximum $l\in\bbZ$ satisfying $h (\sum_{k=-\infty}^{l-1} x'(kh))/ \pi \le \eps/2$
      \State Find the minimum $r\in\bbZ$ satisfying $4\pi(1+\sqrt{2})(\sum_{k=r+1}^\infty ku(kh)) / |\sgm| \le \eps/2$
      \State \Return $l, r$
    \EndFunction
  \end{algorithmic}
  \label{alg:1}
\end{algorithm}

Algorithm \ref{alg:1} requires the shift parameter $\sgm$, and this parameter must be chosen appropriately.
Indeed, when $\sgm$ is positive, the integral representation \eqref{eq:ir_rme} does not make sense.
In addition, when $\sgm$ is large in the negative direction, it makes inaccurate results because the condition number for $\rme^A$ becomes large as $\|A\|$ is large, see, e.g., \cite[Sect.\ 1]{guttel2016scaled}.
From numerical examples in Subsection \ref{sect:ex_shift}, it may be better to choose $\sgm$ from $[-5, 0)$.

The computation of the resolvent $(x_h(kh)^2I + A^2)^{-1}$ is necessary in Step 4 of Algorithm \ref{alg:1}.
The condition number for the resolvent can be as large as the square of the condition number of $A$, and it may result in low accuracy.
In this case, although the computational cost increases, we can compute as
\begin{align}
  (x_h(kh)^2I + A^2)^{-1} = \frac{\rmi}{2x_h(kh)}\left[(\rmi x_h(kh)I+A)^{-1} - (-\rmi x_h(kh)I+A)^{-1}\right]
\end{align}
for the accuracy.

Infinite sums in steps 6 and 7 of Algorithm \ref{alg:1} can be approximated with not too large computations because the summands $x_h'(kh)$ and $ku(kh)$ decays double exponentially.
For example, Figure \ref{fig:summand} illustrates the value of $|x_h'(kh)|$ and $|ku(kh)|$ for $h=0.05$, and it shows that these sums can be estimated within the 50 computations of the summands respectively.

\begin{figure}[htbp]
  \centering
  \includegraphics{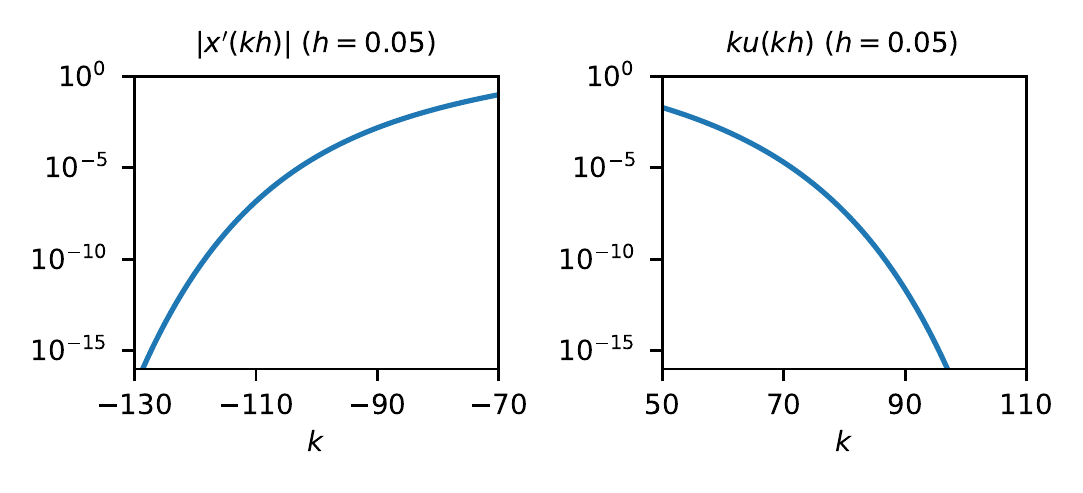}
  \caption{The absolute value of the summands in the upper bound on the truncation error. The left figure illustrates $|x'_h(kh)|$ in \eqref{eq:ub_trunc_err}, and the right one illustrates $|ku(kh)|$.}
  \label{fig:summand}
\end{figure}

We next propose an automatic quadrature algorithm.
Ooura and Mori proposed an automatic quadrature algorithm in \cite[Sect.\ 5]{ooura1999robust} that is applicable when the convergence rate of the DE formula, the constant $\rho>0$ such that $|\calI - h\sum_{k=-\infty}^\infty x_h'(kh) g(x_h(kh))\sin(x_h(kh))| = \calO(\exp(-\rho h))$, is known.

In the computation of $\rme^A$, the error $\|\rme^A - h\sum_{k=-\infty}^\infty F_h(kh)\|$ can be bounded by $\gamma\exp(-\rho/h)$ with some positive constants $\gamma$ and $\rho$ because any element of $F_h(t)$ can be represented as a rational function of $t$.
However, the convergence rate $\rho$ varies with the eigenvalue distribution of $A$ because the singularities of the integrand in \eqref{eq:ir_rme_de} are solutions $z$ of $x_h(z)^2 = \lambda^2$, where $\lambda$ is any eigenvalue of $A$.
Thus, we add an estimation of the convergence rate to the algorithm.
Once we estimate $\gamma$ and $\rho$, we can select an appropriate mesh size $h$ by solving $\eps = \gamma\exp(-\rho/h)$ for $h$ with a given tolerance $\eps$.

In the algorithm, we first select three mesh sizes $h_1 > h_2 > h_3 > 0$.
Let $X_i$ $(i=1,2,3)$ be the computational result of the DE formula \eqref{eq:ir_rme_de} with the mesh size $h_i$.
If $X_3$ is much more accurate than $X_1$ and $X_2$, the errors of $X_1$ and $X_2$ can be approximated as
\begin{align}
  \|X_i - \rme^A\| \approx \eps_i := \|X_i - X_3\| \qquad (i=1,2).
\end{align}
By solving $\eps_i = \gamma\exp(-\rho/h_i)$ $(i=1,2)$ for $\rho$ and $\gamma$, we get
\begin{align}
  \rho = \frac{h_1h_2}{h_1 - h_2} \log\left(\frac{\eps_1}{\eps_2}\right),
  \qquad
  \gamma = \eps_1\exp\left(\frac{\rho}{h_1}\right),
\end{align}
and we would approximate $\|X_3 -\rme^A\| \approx \eps_3 := \gamma\exp(-\rho/h_3)$.
Because $\eps_3$ is just an approximation, we would select a safety parameter $\eta > 0$ and stop the algorithm if $\eps_3 \le \eta\eps$ to obtain $X_3(\approx \rme^A)$.
If $\eps_3 > \eta\eps$, we set $h_4 = \rho/\log(\gamma/\eta\eps)$ and compute the DE formula \eqref{eq:ir_rme_de} once more.
When $\eps_1 \lesssim \eps_2$, it means that the initial mesh size is too large to assume the exponentially convergence of the DE formula \eqref{eq:ir_rme_de}.
For this case, we set $h_i \gets h_{i+1} ~ (i=1,2)$ and repeat the procedure.
The detail of the algorithm is given in Algorithm \ref{alg:2}.

\begin{algorithm}
  \caption{Automatic quadrature algorithm for $\rme^A$ based on the DE formula}
  \begin{algorithmic}[1]
    \Statex \textbf{Input} $A\in\bbC^{n\times n}$, tolerance for the error $\eps>0$, shift parameter $\sgm < 0$, initial mesh size $h_1 > 0$, safety parameter $\eta > 0$, the minimum mesh size $\hmin > 0$
    \Statex \textbf{Output} $X \approx \rme^A$
    \State Compute the rightmost eigenvalue of $A$, $\lmdr$
    \State $\tldA = A + (\sgm - \lmdr)I$
    \State $h_2 = h_1/2, \quad h_3 = h_1/4$
    \State $l_i, r_i = \texttt{GetInterval}(\sigma,\eps/2,h_i) \quad (i=1,2,3)$
    \Comment Set $\epsilon/2$ for the input of \texttt{GetInterval} to bound the truncation error by $\epsilon/2$
    \State $\tldX_i = h_i \sum_{k=l_i}^{r_i} F_{h_i}(kh_i,\tldA) \quad (i=1,2,3)$
    \State $\eps_i = \|\tldX_i - \tldX_3\|_2 \quad (i=1,2)$ \label{alg2:compute_lr}
    \State $\rho = h_1h_2\log(\eps_1/\eps_2)/(h_1-h_2), \quad \gamma = \eps_1\exp(\rho/h_1)$
    \State $\eps_3 = \gamma\exp(-\rho/h_3)$
    \If {$\eps_3 < \eta\eps$}
        \State \textbf{return} $X = \rme^{\lmdr - \sgm}\tldX_3$
    \Else
      \State $h_4 = \rho / \log(\gamma/\eta\eps)$
      \If {$\eps_1 \lesssim \eps_2$ or $h_4 < \hmin$}
        \State Set $h_i \gets h_{i+1}$ $(i=1,2,3)$ and $\tldX_i \gets \tldX_{i+1}$ $(i=1,2)$
        \State $l_3, r_3 = \texttt{GetInterval}(\sigma,\eps/2,h_3)$
        \State $\tldX_3 = h_3 \sum_{k=l_3}^{r_3} F_{h_3}(kh_3,\tldA)$
        \State go to Step \ref{alg2:compute_lr}.
      \Else
        \State $l_4, r_4 = \texttt{GetInterval}(\sigma,\eps/2,h_4)$
        \State $\tldX_4 = h_4 \sum_{k=l_4}^{r_4} F_{h_4}(kh_4,\tldA)$
        \State \textbf{return} $X = \rme^{\lmdr - \sgm}\tldX_4$
      \EndIf
    \EndIf
  \end{algorithmic}
  \label{alg:2}
\end{algorithm}

\section{Numerical examples} \label{sect:experiment}
The computation is carried out with Julia 1.9.1 on an Intel Core i5-9600K CPU with 32 GB RAM.
The IEEE double-precision arithmetic is used unless otherwise stated.
The reference solutions in Subsection \ref{sect:ex_comparison} are computed by $[13/13]$-type Pad\'e approximation with the \texttt{BigFloat} data type of Julia, which gives roughly 77 significant decimal digits.
The Pad\'e approximation is performed after scaling so that the 1 norm of the matrix is less than 5.37.
The programs for these experiments are available on GitHub\footnote{\url{https://github.com/f-ttok/article-expmde/}}.
For the DE formula, we use the change of variables in \eqref{eq:trans_99}, and the infinite sums in \texttt{GetInterval} are truncated with 50 terms.

\subsection{Computing the scalar exponential with the DE formula} \label{sect:ex_scalar}
When the numerical range of $A$ is known, the upper bound on the error of the DE formula can be obtained by $\|\rme^A - h\sum_{k=l}^r F_h (kh,A)\|_2 \le (1+\sqrt{2}) \max_{z \in \calW(A)} |\rme^z - h\sum_{k=l}^r F_h(kh,z)|$.
Hence, it would be worth to observe the error of the DE formula for the scalar exponential $|\rme^z - h\sum_{k=l}^r F_h(kh,z)|$.
Figure \ref{fig:err_de_scalar} visualizes the error, and we can see the followings:
\begin{itemize}
  \item When $z$ is on or near the negative real axis, the error of the DE formula with $h=0.1$ is about $10^{-16}$.
  \item When $-\Re(z)$ is large, the error of the DE formula with $h=0.2$ is about $10^{-16}$.
  \item When $z$ is in a sector region $\{z\in\bbC\colon |\arg(-z)| < \pi/4\}$, the error of the DE formula with $h=0.05$ is about $10^{-16}$.
\end{itemize}
From these findings, when $A$ is a near Hermitian matrix, the error in 2-norm of the DE formula for $\rme^A$ with $h = 0.1$ will be about $10^{-16}$.
\begin{figure}[tbp]
  \includegraphics[width=\linewidth]{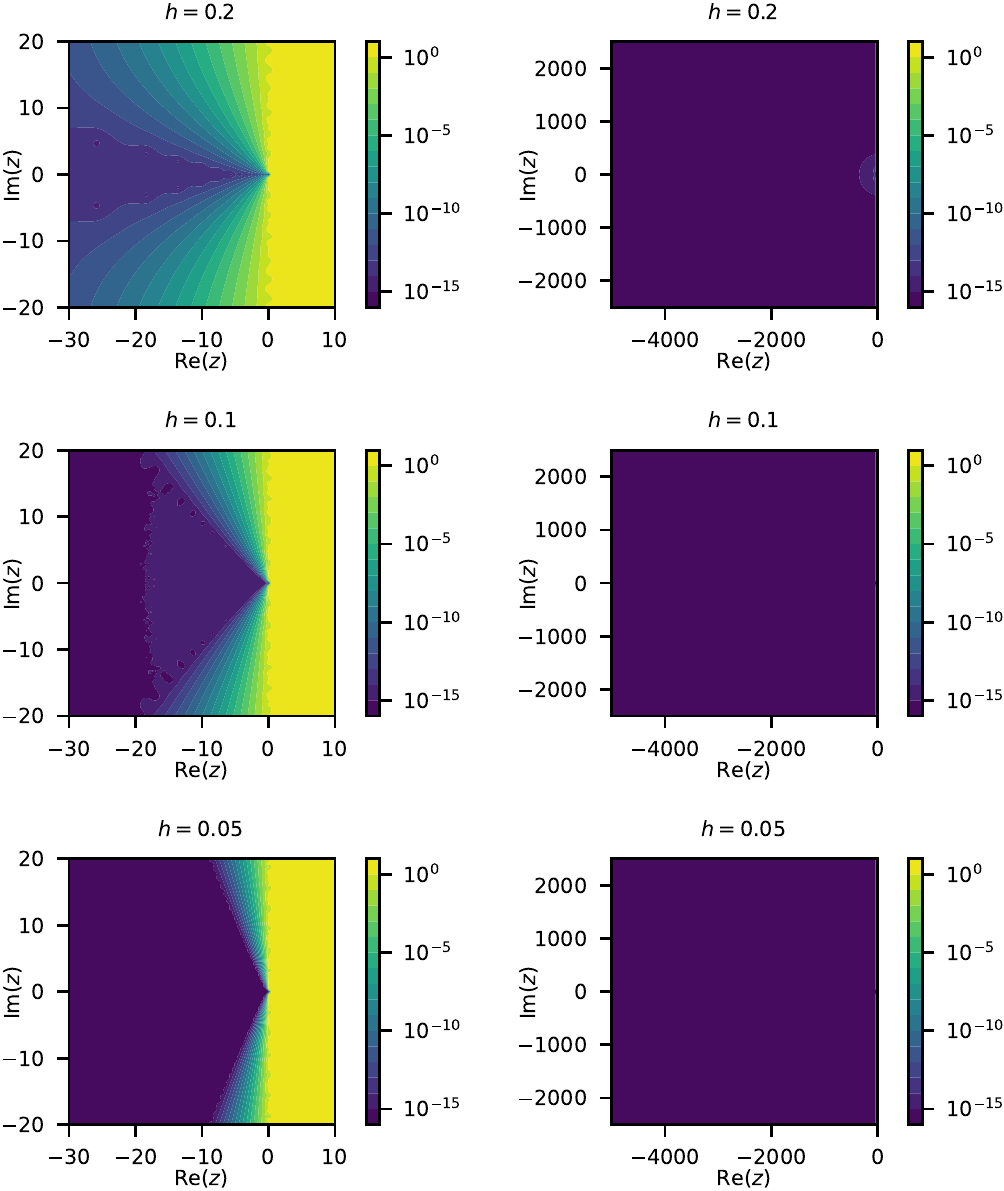}
  \caption{%
    The error of the double exponential formula for the scalar exponential $\rme^z$.
    The mesh size is set to $h=0.2, 0.1, 0.05$.
    The left column shows the error for the area $\{z\in\bbC\colon -30 \le \Re(z) \le 10,\ -20 \le \Im(z) \le 20\}$ and the right one shows the error for $\{z\in\bbC\colon -5000 \le \Re(z) \le 0,\ -2500 \le \Im(z) \le 2500\}$.
  }
  \label{fig:err_de_scalar}
\end{figure}

\subsection{Accuracy dependence on the shift parameter \texorpdfstring{$\sgm$}{}}\label{sect:ex_shift}
Algorithm \ref{alg:1} requires the parameter $\sigma$ for the input.
For reference, we see the accuracy dependence of the DE formula on $\sigma$.
We consider two test matrices $A_k = ZD_kZ^{-1}$ $(k=1,2)$, where $Z \in\bbC^{50\times 50}$ is generated by using a function \texttt{randsvd} in MatrixDepot.jl \cite{zhang2016matrix} so that $\kappa(Z) = 10^2$, and $D_k \in\bbC^{50\times 50}$ is a diagonal matrix whose diagonals are $d_i = 1 -10^{2k(i-1)/49} + \rmi \nu_i/20 \ (i=1,\dots,50, \ \nu_i \sim \calN(1,0))$.

Figure \ref{fig:ex_shift_err} shows the error of the DE formula for $\sigma \in [-10,5]$.
When $\sigma \ge 0$, the error of the DE formula is large for both matrices because $A$ does not satisfy the condition for \eqref{eq:ir_rme}.
Thus, $\sigma$ must be smaller than 0.
On the other hand, as $\sigma$ becomes small, the computational result becomes inaccurate.
Hence, it would be better to set $\sigma \in [-5,0)$.

\begin{figure}[htbp]
  \centering
  \includegraphics[scale=1.5]{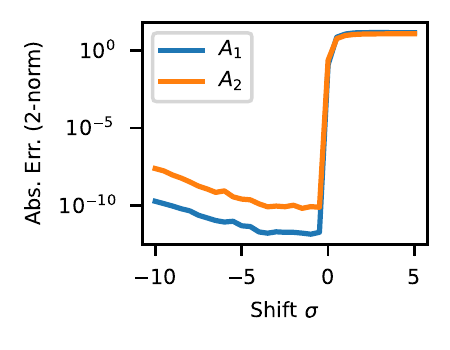}
  \caption{Accuracy dependence of the double exponential formula on the shift parameter $\sigma$.}
  \label{fig:ex_shift_err}
\end{figure}

\subsection{Accuracy of automatic quadrature}
Here, we see the accuracy of Algorithm \ref{alg:2}.
Figures in Figure \ref{fig:ex_automatic_quadrature} show the error of Algorithm \ref{alg:2}, where the test matrices are the same as in Subsection \ref{sect:ex_shift}, and the safe parameter is set to $\eta = 10$.
The left figure shows the accuracy dependence on the input tolerance $\eps$, and it shows that the accuracy of Algorithm \ref{alg:2} could be controlled by $\eps$.
The right figure shows the error estimates of the computational results for the selected mesh size for the case where $A = A_1$ and $\eps=10^{-10}$.
From the figure, it is found that the algorithm could estimate the error.

\begin{figure}[htbp]
  \centering
  \includegraphics{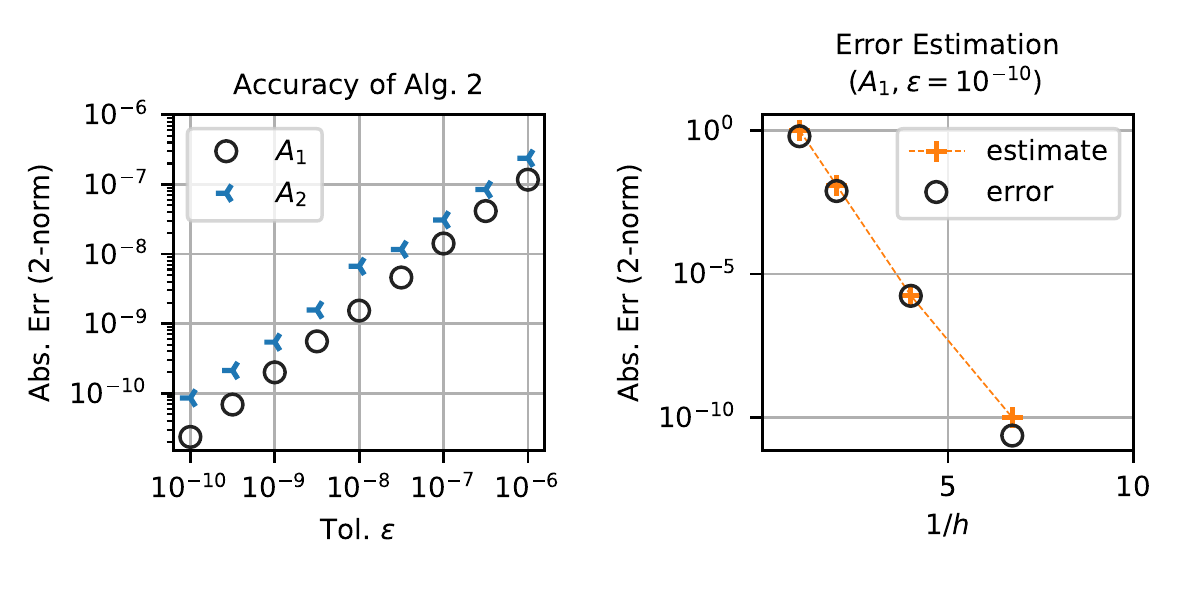}
  \caption{The accuracy of Algorithm \ref{alg:2}.
    The left figure shows the accuracy dependence on the input tolerance $\eps$ and the error of the computational results.
    The right figure shows the Error of the case when $A = A_1$ and $\eps=10^{-10}$.
    The horizontal axis is the inverse of the selected mesh size, and the vertical axis is the error.
    }
  \label{fig:ex_automatic_quadrature}
\end{figure}

\subsection{A comparison of the DE formula with the Cauchy integral for non-Hermitian matrices}\label{sect:ex_comparison}
Using non-Hermitian matrices, we compare Algorithm \ref{alg:1} (denoted by \texttt{DE}) with an algorithm based on the Cauchy integral.
The matrices are generated from the convection-diffusion problems
\begin{alignat}{2}
  & \frac{\partial u}{\partial t} = d \Delta u - \bm{c}^\top \nabla u
  \qquad && \text{in} \quad \Omega = (0,1)^2,\\
  & u = 0 && \text{on} \quad \partial \Omega,\\
  & u(0,\bmx) = u_0(\bmx),
\end{alignat}
where $d = 0.001, 0.01$, and $ \bm{c} = [0.2, 0.2], [0.4, 0.4]$.
By using the finite element method, we have a linear evolution equation $M\bmu'(t) = K\bmu(t)$ whose solution is $\exp(tM^{-1}K)\bmu(0)$.
We use FreeFEM++ \cite{hecht2012freefem} for the discretization, and the eigenvalues of $M^{-1}K$ is illustrated in Figure \ref{fig:ex_cvp_eigvals}.

For \texttt{DE}, we set $\sigma=-2.5$ and $\eps \approx 2.2\times 10^{-16}$.
For the Cauchy integral
\begin{align}
  \rme^A = \frac{1}{2\pi\rmi} \int_\Gamma \rme^z (zI + A)^{-1} \dd{z}
\end{align}
where $\Gamma$ is a contour enclosing the eigenvalue of $A$,
we used the algorithm (denoted by \texttt{Talbot}) proposed in \cite{dingfelder2015improved}.
\texttt{Talbot} employs the midpoint rule on the Talbot contour
\begin{align}
  \Gamma = \left\{m\Bigl(-\tilde{\sigma} + \tilde{\mu}\theta \cot(\tilde{\alpha}\theta) + \tilde{\nu}\rmi\theta \Bigr) \colon -\frac{\pi}{2} \le \theta \le -\frac{\pi}{2} \right \},
\end{align}
where $m$ is the number of abscissas, and $\tilde{\sigma},\tilde{\mu},\tilde{\alpha},\tilde{\nu}$ is selected to minimize the error of the midpoint rule.
Because \texttt{Talbot} focuses on Hermitian matrices, these numerical results would be out of the scope of the algorithm.

\begin{figure}[tbp]
  \centering
  \includegraphics{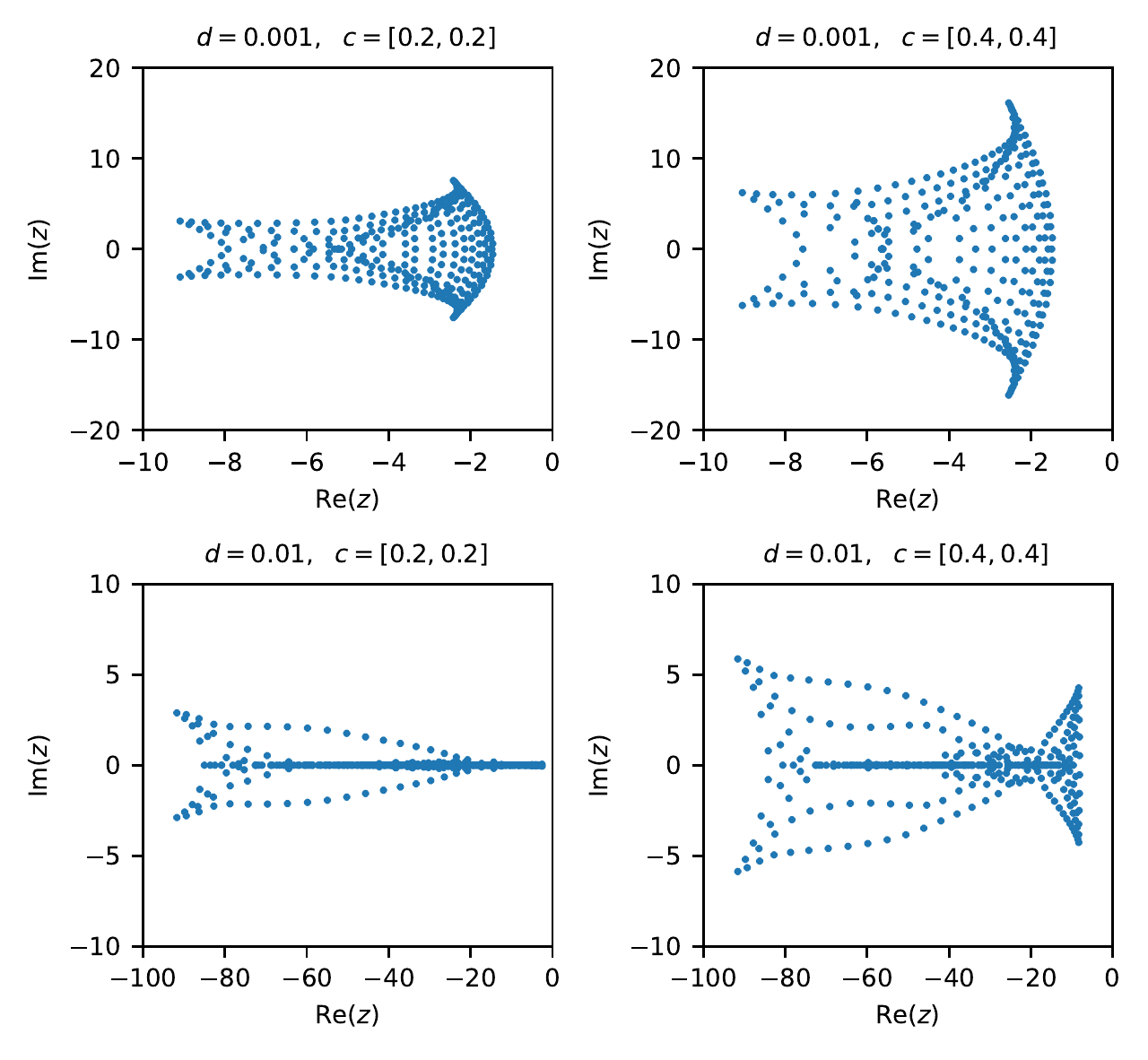}
  \caption{Eigenvalues of the test matrix $M^{-1}K$ in Subsection \ref{sect:ex_comparison}.}
  \label{fig:ex_cvp_eigvals}
\end{figure}

Figure \ref{fig:ex_cvp_histories} illustrates the convergence histories of these algorithms for $t=1$.
For $d=0.01$, both \texttt{DE} and \texttt{Talbot} give accurate results, and for $d=0.001$, \texttt{DE} still gives accurate results while \texttt{Talbot} does not converge.
Hence our algorithm would be a choice of quadrature-based algorithms for non-Hermitian matrices.
\begin{figure}[tbp]
  \centering
  \includegraphics{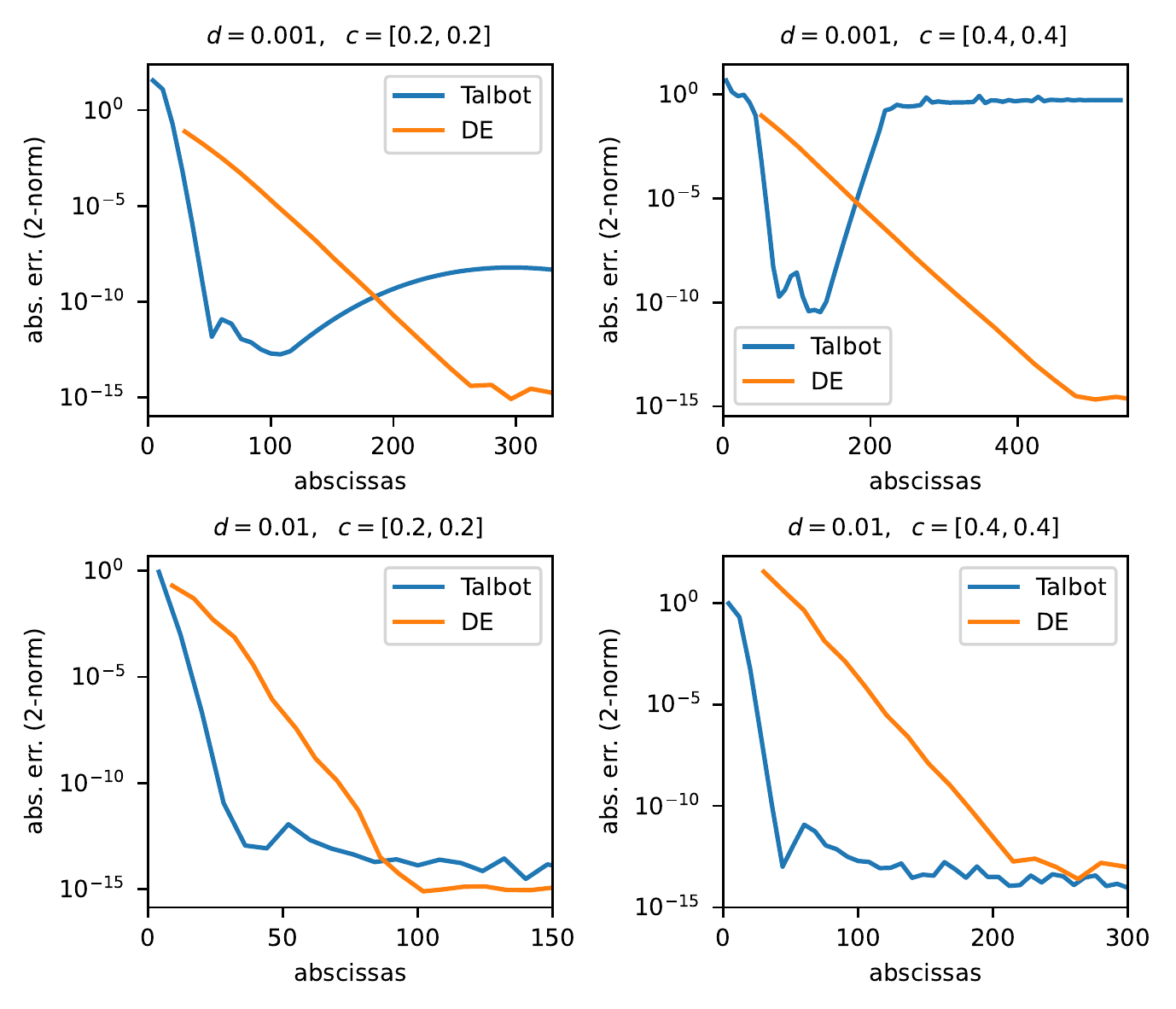}
  \caption{Convergence histories of \texttt{DE} and \texttt{Talbot} for the problem $\exp(M^{-1}K)$ in Subsection \ref{sect:ex_comparison}.}
  \label{fig:ex_cvp_histories}
\end{figure}

\section{Conclusion} \label{sect:conclusion}
In this paper, the DE formula was considered for the computation of $\rme^A$.
To utilize the DE formula, we analyzed the truncation error and proposed algorithms.
Numerical results showed the validity of the algorithm.

Future work includes improvement of the change of variables of the DE formula for $\rme^A$ and comprehensive performance evaluations of the proposed algorithms on parallel computers for practical problems.

\section*{Acknowledgments}

This work was supported by JSPS KAKENHI Grant Number 20H00581.
The authors would like to thank Enago (www.enago.jp) for English language editing.

\appendix

\section{Derivative of the change of variables} \label{sect:derivative_cv}
Let $v(t) = -2t - \alpha(1-\rme^{-t}) - \beta(\rme^t-1)$.
Then, the transformation \eqref{eq:trans_99} for the DE formula is $x_h(t) = \pi t/h (1 - \exp(v(t)))$ and its derivative is
\begin{align}
  x_h'(t) = \frac{\pi}{h} \frac{1 - \rme^{v(t)} + tv'(t)\rme^{v(t)}}{(\rme^{v(t)} - 1)^2}.
\end{align}
Because $\exp(v(t)) \to 1$ as $t\to 0$, we would use $x_h(0) = \pi / h(\alpha+\beta+2)$ and
\begin{align}
  x_h'(0) = \frac{\pi}{2h} \frac{
    \alpha^2 + 2\alpha\beta + 5\alpha + \beta^2 + 3\beta + 4
  }{
    \alpha^2 + 2\alpha\beta + 4\alpha + \beta^2 + 4\beta + 4
  }.
\end{align}


\end{document}